\documentclass[11pt]{amsart}
\usepackage[pagewise]{lineno}
\usepackage{amssymb,amsmath,latexsym, amscd, graphicx, url}
\newtheorem{theorem}{Theorem}[section]
\newtheorem{lemma}[theorem]{Lemma}
\newtheorem{proposition}[theorem]{Proposition}
\newtheorem{definition}[theorem]{Definition}

\newtheorem{corollary}[theorem]{Corollary}

\theoremstyle{definition}
\newtheorem{remark}[theorem]{Remark}
\newtheorem{example}[theorem]{Example}
\newtheorem{question}[theorem]{Question}

\begin{document}

\title{Generalizations of the Burns-Hale Theorem}

\begin{abstract}
The Burns-Hale theorem states that a group $G$ is left-orderable if and only if $G$ is locally projectable onto the class of left-orderable groups.  Similar results have appeared in the literature in the case of UPP groups and Conradian left-orderable groups, with proofs using varied techniques in each case.

This note presents a streamlined approach to showing that if $\mathcal{C}$ is the class of either Conradian left-orderable, left-orderable, or UPP groups, then $\mathcal{C}$ contains all groups that are locally projectable onto $\mathcal{C}$; and shows that this streamlined approach works for the class of diffuse groups as well.  It also includes an investigation of the extent to which a similar theorem can hold for the classes of bi-orderable, circularly orderable or recurrent orderable groups.
\end{abstract}
\author[Adam Clay]{Adam Clay}
\thanks{Adam Clay was partially supported by NSERC grant RGPIN-2014-05465}
\address{Department of Mathematics, 420 Machray Hall, University of Manitoba, Winnipeg, MB, R3T 2N2.}
\email{Adam.Clay@umanitoba.ca}
\urladdr{http://server.math.umanitoba.ca/~claya/}

\maketitle

A group $G$ is called left-orderable (LO) if it admits a strict total ordering $<$ such that $g<h$ implies $fg<fh$ for all $f, g,h \in G$.  Closely related is the notion of a group being bi-orderable (BO), which is when a group admits a left-ordering for which $g<h$ additionally implies $gf<hf$.  In a bi-ordered group, positive elements are conjugation invariant---meaning $id<g$ implies $id<h^{-1}gh$ for all $h \in G$.  By weakening this condition and instead requiring that $id <g$ implies $id<h^{-1}gh^2$ for all positive $h \in G$, we arrive at Conradian left-orderings and Conradian left-orderable groups (CO).\footnote{This paper considers both Conradian orderings and circular orderings, but the notation ``CO" will be reserved for Conradian-orderable groups.}  Every bi-ordering is evidently a Conradian left-ordering. 

There is a fourth type of ordering, called a recurrent ordering (RO), which is a special type of Conradian ordering.  These orderings arise naturally from considering amenable left-orderable groups \cite{Morris06}, and more generally from any group $G$ whose action on the space of left-orderings $\mathrm{LO}(G)$ admits a recurrent point.  A left-ordering $<$ of a group $G$ is called \textit{recurrent} (or recurrent for every cyclic subgroup, as in \cite{Morris06}) if for every $g \in G$ and for every finite sequence of inequalities 
\[ g_1 < g_2 <\ldots < g_n
\]
there exists an increasing sequence $\{n_k \}_{k=1}^{\infty}$ of positive integers such that 
\[ g_1g^{n_k} < g_2 g^{n_k} < \ldots < g_n g^{n_k}
\]
for all $k$.  Every bi-ordering is evidently a recurrent ordering.

Weaker than the notion of left-orderability of a group $G$ is the notion of $G$ having the unique product property (UPP), and being diffuse (see Subsections \ref{UPP} and \ref{diffuse}).  There is a chain of implications
\[ \mbox{BO} \Rightarrow \mbox{recurrent orderable} \Rightarrow \mbox{CO} \Rightarrow \mbox{LO} \Rightarrow \mbox{diffuse} \Rightarrow \mbox{UPP} \Rightarrow \mbox{torsion free}.
\]

An important tool in the study of left-orderable groups is the Burns-Hale theorem \cite{BH72}, which in its classical form is as follows:
\begin{theorem}
\label{original BH}
A group $G$ is left-orderable if and only if for every nontrivial finitely generated subgroup $H$ of $G$ there exists a homomorphism $H \rightarrow L$ onto a nontrivial left-orderable group $L$.
\end{theorem}
This is essential, for instance, in showing that fundamental groups of certain nice $3$-manifolds are left-orderable whenever they have a left-orderable quotient \cite{BRW05}, or in showing that the group of PL homeomorphisms of the disk is left-orderable \cite{CR15}.

Of the properties mentioned above, the literature contains proofs that CO, LO, and UPP groups satisfy a theorem similar to the Burns-Hale theorem, though the proofs are somewhat varied in flavour.  The purpose of this note is to show how CO, LO, diffuse and UPP groups all satisfy the same ``Burns-Hale theorem" (with similar proofs in all cases); and also to investigate the extent to which BO groups, recurrent orderable groups and circularly orderable groups satisfy something akin to the Burns-Hale theorem as well.

\section{Classical Burns-Hale type theorems}
Let $P$ be a property of a finite subset of a group, we suppose that the empty set always has property $P$.  We say that property $P$ \textit{respects extensions} if for every nonempty finite subset $X \subset G$ and every short exact sequence 
\[ 1 \rightarrow K \rightarrow \langle X \rangle \stackrel{q}{\rightarrow} H \rightarrow 1
\]
where $H$ is nontrivial, the following holds: If all finite subsets $Y \subset K$ with $|Y|<|X|$ have property $P$ and $q(X)$ has $P$, then $X$ has property $P$.  Note that the cardinality restriction on $Y$ may seem artificial, but will serve as the key to an inductive step in later proofs.

 Let $P$ be a property of finite subsets of a group that respects extensions.  Let $\mathcal{C}$ be a class of groups defined by $G \in \mathcal{C}$ if and only if every nonempty finite subset of $G$ has property $P$, in which case we will say that \textit{$\mathcal{C}$ has local property $P$}.  We call a group $G$ \textit{locally projectable to $\mathcal{C}$} if for every nontrivial finitely generated subgroup $F \subset G$ there is a nontrivial group $H \in \mathcal{C}$ and a surjective homomorphism $\phi :F \rightarrow H$.
 
 Then the following Burns-Hale type theorem holds (cf. \cite[Theorems 1 and 2]{BH72}):
 
 \begin{theorem}
 \label{classical BH}
 Let $G$ be a group, $P$ a property of finite subsets of a group that respects extensions, and $\mathcal{C}$ a class of groups having local property $P$.  Then $G$ is locally projectable to $\mathcal{C}$ if and only if $G \in \mathcal{C}$.
 \end{theorem}
\begin{proof}

Suppose that $G$ is locally projectable to $\mathcal{C}$, but that $G \notin \mathcal{C}$.  Then there is a smallest nonempty finite subset $X \subset G$ that does not have property $P$.

Since $G$ is locally projectable to $\mathcal{C}$, there is a short exact sequence
\[ 1 \rightarrow K \rightarrow \langle X \rangle \rightarrow H \rightarrow 1
\]
where $H \in \mathcal{C}$, and since $|X|$ is minimal every $Y \subset K$ with $|Y|<|X|$ has property $P$.   This contradicts the fact that property $P$ respects extensions.
%
%
\end{proof}

This version of the Burns-Hale theorem accounts for all of its incarnations throughout the literature: it applies to CO groups, LO groups, diffuse groups and UPP groups.  These applications are reviewed below, showing them all to be instances of the same principle.  As it seems that a Burns-Hale type theorem for diffuse groups has not appeared in the literature before, it is covered in more detail than the others.

\subsection{CO groups.} \cite[Proposition 3.11]{Navas10}, \cite[Theorem 9.17]{CR16} Given a finite set $X \subset G$, let $C(X)$ denote the smallest subsemigroup of $G$ satisfying $X \subset C(X)$ and $x^{-1}yx^2 \in C(X)$ for all $x, y \in C(X)$.  Then we have:

\begin{theorem}
\label{CO theorem}
 A group $G$ admits a Conradian left-ordering if and only if for every finite subset $X \subset G$ with $X \setminus \{ id \} = \{x_1, \ldots, x_n\}$ there exist exponents $\epsilon_i = \pm 1$ for $i = 1, \ldots, n$ such that $id \notin C(x_1^{\epsilon_1}, \ldots, x_n^{\epsilon_n})$.
\end{theorem}

One can verify that the property given in Theorem \ref{CO theorem} respects extensions, this is implicit in \cite[Lemma 9.18]{CR16} and \cite[Theorem H, pp. 66-67]{CH98}\footnote{The author attributes this theorem to Tararin.}. Thus the class of CO groups obeys a Burns-Hale type theorem.  

Since every finitely generated Conradian left-orderable group admits a homomorphism onto the integers, the Burns-Hale theorem for Conradian left-orderable groups has more often appeared in the literature as follows:

\begin{theorem} \cite{Brodskii84, RR02, Navas10, CH98}
A group is Conradian left-orderable if and only if it is locally indicable.
\end{theorem}

\subsection{LO groups.} 
Given a finite subset $X \subset G$, let $S(X)$ denote the semigroup generated by $X$.
\begin{theorem} \cite[Theorem 2.2]{Conrad59}
A group $G$ is left-orderable if and only if for every finite subset $X \subset G$ with $X \setminus \{ id \} = \{ x_1, \ldots, x_n\}$ there exist exponents $\epsilon_i = \pm 1$ for $i = 1, \ldots, n$ such that $id \notin S( x_1, \ldots, x_n)$.
\end{theorem}
It is easy to verify that this property respects extensions, and so we arrive at the classical Burns-Hale theorem (Theorem \ref{original BH}), which is as it appears in \cite[Theorem 2]{BH72}.

\subsection{Diffuse groups.}
\label{diffuse}
The notion of a diffuse group was first introduced by Bowditch \cite{Bowditch00} as a generalization of the unique product property.  

Let $G$ be a torsion free group.  Given a finite subset $A$ of $G$, an extreme point of $A$ is $a \in A$ such that $a^{-1} A \cap A^{-1} a = \{ id \}$.  Here, $A^{-1} = \{ a^{-1} \mid a \in A \}$.   A group $G$ is called \textit{weakly diffuse} if every finite subset of $G$ has an extreme point.  A group is called \textit{diffuse} if every finite subset $A$ with $|A| >1$ has two extreme points.

Recall that an ordering of a group $G$ (partial or total) is \textit{locally invariant} if for all $x, y \in G$ with $y \neq 1$, either $xy > x$ or $xy^{-1} > x$.

\begin{theorem} \cite[Proposition 6.2]{LM12} For any group $G$, the following are equivalent:
\label{all the same}
\begin{enumerate}
\item $G$ is weakly diffuse.
\item $G$ is diffuse.
\item $G$ admits a locally invariant partial ordering.
\item $G$ admits a locally invariant total ordering.
\end{enumerate}
\end{theorem}

For some time it was unknown whether or not the above properties were also equivalent to left-orderability of $G$, but Nathan Dunfield has recently produced an example of a group which is not LO, but is diffuse \cite[Appendix]{KR16}.  It is unknown whether or not these properties are equivalent to the unique product property.

%
%
%
%

\begin{theorem}
Let $G$ be a group.  Then $G$ is diffuse if and only if for every nontrivial finitely generated subgroup $F$ of $G$ there exists a homomorphism $F \rightarrow H$ onto a nontrivial diffuse group.
\end{theorem}
\begin{proof}
One can check that the defining property of weakly diffuse groups is equivalent to the following:  Every finite set containing the identity admits an extreme point.  We will show that this property respects short exact sequences.

To this end, suppose that $X \subset G$ is a finite subset containing the identity, and that 
\[ 1 \rightarrow K \rightarrow \langle X \rangle \stackrel{q}{\rightarrow} H \rightarrow 1
\]
is a short exact sequence where $H$ is a nontrivial diffuse group.   Suppose that for every set $Y \subset K$ with fewer than $|X|$ elements, the set $Y$ admits an extreme point, and that the set $q(X)$ admits an extreme point.

Let $q(a)$ be an extreme point of $q(X)$, and note that $1 \leq |a^{-1}X \cap K| \leq |X|$ since at least one point of $a^{-1}X$ is not in $K$, and $id \in a^{-1}X \cap K$.   Choose an extreme point $b$ of $a^{-1}X \cap K$.  We claim that $ab$ is an extreme point for $X$, which will complete the proof.

First note that  $ab \in X$ since $b \in a^{-1}X$.  Now let $h \in b^{-1} a^{-1} X \cap X^{-1} ab$ be given, we show that $h=id$.  Applying the homomorphism $q$ and recalling that $b \in K$, we get
\[ q(h) \in q(a)^{-1} q(X) \cap q(X)^{-1}q(a)
\]
which implies that $q(h) =id$ since $q(a)$ is an extreme point of $q(X)$.

Now $h \in K$, and so $h \in (b^{-1} a^{-1} X \cap X^{-1} ab) \cap K$, or in other words
\begin{align*}
h &\in (b^{-1} a^{-1} X \cap K ) \cap (X^{-1} ab \cap K ) \\ 
&= b^{-1} (a^{-1} X \cap K ) \cap (X^{-1} a \cap K )b \\ 
&=b^{-1} (a^{-1} X \cap K ) \cap(a^{-1} X \cap K )^{-1} b .
\end{align*}
thus $h =id$, since $b$ is an extreme point of $a^{-1}X \cap K$.  The result now follows by Theorem \ref{classical BH}.

\end{proof}

\subsection{UPP groups.}
\label{UPP}
A group $G$ is said to have the \textit{unique product property} (we say ``$G$ is UPP" for short) if for every pair $(A, B)$ of finite subsets of $G$ there exists at least one pair $(a,b) \in A \times B$ such that if $ab = a'b'$ where $a' \in A$ and $b' \in B$ then $a=a'$ and $b=b'$.  The element $ab$ is called \textit{a unique product for $AB$}. 


Equivalently, a group $G$ is UPP if and only if it satisfies the following property on finite subsets:  For every finite subset $A$ of $G$, whenever $X, Y \subset A$ with $|X|+|Y| \leq |A|$ and $\{ id \} \subset X \cap Y$ then there is a unique product $xy$ for $XY$.

To see that this condition implies UPP, take two subsets $X$ and $Y$ of $G$ that do not satisfy $\{ id \} \subset X \cap Y$.  Choose $x \in X$ and $y \in Y$, and let $(x^{-1}g)(hy^{-1})$ be a unique product for $(x^{-1}X)(Yy^{-1})$.  Then one checks that $gh$ is a unique product for $XY$.

\begin{theorem} \cite[Theorem 1]{BH72} and \cite[Lemma 1.8 (iii)]{Passman77}
Let $G$ be a group.  Then $G$ is UPP if and only if for every nontrivial finitely generated subgroup $F$ of $G$ there exists a homomorphism $F \rightarrow H$ onto a nontrivial UPP group.
\end{theorem}
\begin{proof}
Let $A$, $X$ and $Y$ be subsets of $G$ as above, suppose there is a short exact sequence
\[ 1 \rightarrow K \rightarrow \langle A  \rangle \stackrel{q}{\rightarrow} H \rightarrow 1
\]
such that every subset $B$ of $K$ with $|B|<|A|$ satisfies the property above, that $H$ is nontrivial and there is a unique product for $q(X)q(Y)$, say $q(x)q(y)$.

First if $X, Y \subset K$ then $|X|+|Y|<|A|$ and the existence of a unique product for $XY$ follows immediately from our assumptions.  So suppose otherwise, and set $S = \{ s \in X \mid q(s) = q(x) \}$ and $T = \{ t \in Y \mid q(t) = q(y) \}$, note that at least one of $|S|<|X|$ or $|T|<|Y|$ holds since one of $q(X)$ or $q(Y)$ contains both $\{id\}$ and at least one nonidentity generator for $H$.  Then $x^{-1}S \cup Ty^{-1} \subset K$ and satisfies $|x^{-1}S \cup Ty^{-1}| \leq |S|+|T| < |A|$ and so by assumption there is a unique product for $(x^{-1}S)(Ty^{-1})$, say $(x^{-1}s)(ty^{-1})$.  

We will show that $st$ is a unique product for $XY$, completing the proof.  For suppose that $st = cd$ for some $c \in X$ and $d \in Y$.  Then $q(c)q(d) = q(s) q(t) = q(x) q(y)$, so that $q(c)= q(x)$ and $q(d) = q(y)$.  But then $c \in S$ and $d \in T$, so that $(x^{-1}s)(ty^{-1}) = (x^{-1}c)(dy^{-1})$ forces $x^{-1}s = x^{-1}c$ and $ty^{-1} = dy^{-1}$, since $(x^{-1}s)(ty^{-1})$ is a unique product for $(x^{-1}S)(Ty^{-1})$.  Thus $s=c$ and $t=d$, so the conclusion follows.
\end{proof}

\section{Nonstandard Burns-Hale variants}

There are some natural classes of groups, related to those in the previous section, for which the Burns-Hale theorem cannot hold.  Most notably the class of bi-orderable groups does not admit a Burns-Hale type theorem, since local indicability of a group $G$ yields only a Conradian left-ordering of $G$ (instead of a bi-ordering, as one would expect if Theorem \ref{classical BH} held for BO groups).    To explain this behaviour we make an observation:

\begin{proposition}
Suppose that $\mathcal{C}$ is a class of groups that is closed under taking subgroups, and that $\mathcal{C}$ satisfies Theorem \ref{classical BH}.   If
\[ 1 \rightarrow K \rightarrow G \stackrel{q}{\rightarrow} H \rightarrow 1
\]
is a short exact sequence where both $K, H \in \mathcal{C}$, then $G \in \mathcal{C}$.
\end{proposition}
\begin{proof}
Suppose that $F$ is a finitely generated subgroup of $G$.  If $q(F)$ is nontrivial, then $q:F \rightarrow q(F)$ provides a surjection of $F$ onto a nontrivial element of $\mathcal{C}$.  Otherwise $F \subset K$, and so $F \in \mathcal{C}$ since $\mathcal{C}$ is closed under taking subgroups.  Since $\mathcal{C}$ satisfies Theorem \ref{classical BH}, $G \in \mathcal{C}$.
\end{proof}

This accounts for why BO groups cannot satisfy Theorem \ref{classical BH} (cf. \cite[Problem 1.23]{CR16}), shows that circularly ordered groups cannot satisfy Theorem \ref{classical BH} (the group $\mathbb{Z}_n \times \mathbb{Z}_n$ is not circularly orderable, for example, since all finite circularly orderable groups are cyclic), nor can groups admitting right-recurrent orderings (cf. \cite[Problem 10.45]{CR16} and Proposition \ref{no extensions}).    We study each of these classes of groups in more detail below.

\subsection{Bi-orderable groups.}
Given a group $G$ and a set $X \subset G$, let $N(X)$ denote the smallest normal subgroup containing $X$.  We say that a subgroup $N \subset G$ is finitely normally generated if $N = N( x_1, \ldots, x_n)$ for some finite set of elements $x_1, \ldots, x_n$ of $G$.  The normal subsemigroup of $G$ generated by $X$ will be denoted  $NS(X)$, note that as a semigroup $NS(X)$ is generated by $\{ gxg^{-1} \mid g\in G, x \in X \}$. 

\begin{proposition} \cite{Fuchs63, Ohnishi52, Los54}
\label{semigroup biord}
A group $G$ is bi-orderable if and only if for every finite subset $\{ x_1, \ldots, x_n \} \subset G$ not containing the identity there exists exponents $\epsilon_i = \pm 1 $ such that $NS(x_1^{\epsilon_1}, \ldots, x_n^{\epsilon_n})$ does not contain the identity.
\end{proposition}

\begin{remark} In \cite[Proposition 1.4]{Navas10}, Navas points out that we can replace $NS(x_1^{\epsilon_1}, \ldots, x_n^{\epsilon_n})$ with the smallest semigroup $S$ containing $x_i^{\epsilon_i}$ for $i=1, \ldots, n$ and closed under the property:  For all $x, y \in S$, both $xyx^{-1}$ and $x^{-1}yx$ are in $S$.  This improvement does not seem to allow one to weaken the hypotheses of the theorem below.
\end{remark}

\begin{theorem} (The Burns-Hale analog for bi-orderable groups)
\label{biorderable_theorem}
A group $G$ is bi-orderable if and only if for every nontrivial finitely normally generated subgroup $N$ of $G$ there exists a group $H$ and a surjective homomorphism $\phi:G \rightarrow H$ satisfying:
\begin{enumerate}
\item $\phi(N)$ is nontrivial, and
\item $\phi(N)$ is bi-orderable, and the bi-ordering is invariant under conjugation by elements of $H$.
\end{enumerate}
\end{theorem}

\begin{proof}
If $G$ is bi-orderable, then the identity homomorphism $G \rightarrow G$ always provides the required homomorphism.

For the other direction, we suppose that $G$ satisfies the hypotheses of the theorem.  We will show that the hypotheses of Proposition \ref{semigroup biord} hold:

\hspace{0.5cm}\parbox{11cm}{
 For every finite subset $\{ x_1, \ldots, x_n \}$ of $G$ not containing $id$, there exist exponents $\epsilon_i = \pm 1$, $i = 1, \ldots, n$ such that $id \notin NS(x_1^{\epsilon_1}, \ldots, x_n^{\epsilon_n})$.}
\vspace{1ex}

We proceed by induction on $n$.  As a base case, suppose that $x \in G$ is not the identity, and choose a group $H$ and a homomorphism $\phi :G \rightarrow H$ such that $\phi( N(x))$ is nontrivial and bi-orderable, and the ordering is invariant under conjugation by elements of $H$.  Set $\epsilon = +1$ and consider an arbitrary element $w$ of $NS(x^{\epsilon})$.  The element $w$ is a product of the form
\[  w = \prod_{i=1}^k g_i xg_i^{-1}
\]
where $g_i \in G$, and so applying $\phi$ gives
\[ \phi(w) = \prod_{i=1}^k \phi(g_i) \phi(x) \phi(g_i)^{-1}
\]
Since $\phi(N(x))$ admits a bi-ordering that is invariant under conjugation by elements of $H$, it follows that in this bi-ordering $\phi(g_i) \phi(x)\phi(g_i)^{-1}$ is the same sign for all $i$. Therefore $\phi(w) \neq id$, and hence $w \neq id$.

Now assume that the hypotheses of Proposition \ref{semigroup biord} hold for every finite subset of $G$ containing $n-1$ or fewer elements, none of which are the identity.  Let $\{ x_1, \ldots, x_n \} \subset G \setminus \{ id\}$ be given.  Choose a group $H$ and a homomorphism $\phi: G \rightarrow H$ such that $\phi(N(x_1, \ldots, x_n))$ is bi-orderable with ordering $<$ that is invariant under conjugation by elements of $H$. 

 Re-index the $x_i$'s if necessary, so that 
\[
\phi(x_i)  \left\{ \begin{array}{ll}
=id \mbox{ if $1 \leq i \leq s$} \\
\neq id \mbox{ if $s+1 \leq i \leq n$}
\end{array} \right.
\]
Note that at least one of $\phi(x_i)$ is not equal to the identity, since the image $\phi(N(x_1, \ldots, x_n))$ is nontrivial.

Choose exponents $\epsilon_i = \pm 1$ as follows. For $i = s+1, \ldots, n$ choose $\epsilon_i$ so that $\phi(x_i^{\epsilon_i}) >1$.  For $i = 1, \ldots, s$, use the induction hypothesis to choose $\epsilon_i$ so that $id \notin NS(x_1^{\epsilon_1} , \ldots, x_s^{\epsilon_s})$.  Now let $w$ be an arbitrary element of $NS(x_1^{\epsilon_1} , \ldots, x_n^{\epsilon_n})$, then
\[ w = \prod_{j=1}^k g_j x_{i_j}^{\epsilon_{i_j}} g_j^{-1}
\]
where $g_j \in G$.  If there exists $j$ such that $i_j >s$, then $\phi(x_{i_j})>id$ and so $ \phi(g_j) \phi( x_{i_j}^{\epsilon_{i_j}}) \phi( g_j^{-1}) > id$.  Therefore
\[ \phi(w) = \prod_{j=1}^k \phi(g_j) \phi( x_{i_j}^{\epsilon_{i_j}}) \phi( g_j^{-1})
\]
is a product of non-negative elements with at least one strictly positive element.  Thus $\phi(w) >id$, and so $w \neq id$.

On the other hand, if $i_j \leq s$ for all $j$, then $w \in NS(x_1^{\epsilon_1} , \ldots, x_s^{\epsilon_s})$, and so $w \neq id$ by the induction hypothesis.
\end{proof}

From this theorem it follows that if $G$ is residually torsion-free nilpotent or residually torsion-free central (a group is residually central if $x \notin [x, G]$ for all nonidentity $x \in G$) then $G$ is bi-orderable.  We can also provide an alternative proof of the following result of Rhemtulla:

\begin{proposition}\cite{Rhemtulla73}
If $G$ is residually $p$-finite for infinitely many primes $p$, then $G$ is bi-orderable.
\end{proposition}
\begin{proof}
Suppose that $G$ is residually $p$-finite for infinitely many primes $\{ p_i \}_{i=1}^{\infty}$, and fix a finitely normally generated subgroup $N(x_1, \ldots, x_n)$ of $G$.  For each $i$, fix a normal subgroup $K_i$ of $G$ that is maximal subject to 
\begin{enumerate}
\item $N \not \subset K_i$ and \item $G / K_i$ is a finite $p_i$-group with quotient homomorphism $\phi_i$.
\end{enumerate}

To see that such a subgroup exists, observe that a finite $p_i$-group always has nontrivial centre.  Therefore if $\phi_i(N) \not \subset Z(G/K_i)$, the quotient $(G/K_i)/Z(G/K_i)$ provides a strictly smaller $p_i$-group satisfying (1) and (2) above.  Taking successive quotients, one eventually reaches the smallest such $p_i$-group, and so a maximal subgroup $K_i$.  Note that when $K_i$ is maximal, subject to (1) and (2), $\phi_i(N)$ is central.

Set $P_i = G / K_i $ and consider the canonical map $$\phi: G \rightarrow \prod_{i=1}^{\infty} P_i $$  arising from the maps $\phi_i$.

Then $\phi$ is surjective and $\phi(N)$ is a finitely generated abelian group contained in the centre of $\prod_{i=1}^{\infty} P_i $.  Say $\phi(N) = \mathbb{Z}^k \oplus K$ where $K$ is torsion. Note that $k>1$ since there exists at least one generator $g_j$ of $N$ whose image $\phi_i(g_j)$ is nontrivial for infinitely many $i$, which implies that $\phi(g_j)$ has infinite order. Let
\[ q : \prod_{i=1}^{\infty} P_i \rightarrow \left(\prod_{i=1}^{\infty} P_i \right) /K 
\]
denote the quotient map, whose image we will call $H$.
  Then $q \circ \phi :G \rightarrow H$ provides the required homomorphism, as $(q \circ \phi)(N)$ is torsion free abelian (and thus bi-orderable), and central (and thus any ordering of $q \circ \phi(N)$ is conjugation invariant).

Since $G$ satisfies the hypotheses of Theorem \ref{biorderable_theorem}, it is bi-orderable.
\end{proof}

\subsection{Recurrent orderings.}
\label{RO}
The definition of recurrent orderability found in the introduction to this paper can be reworded as follows: for every finite set $\{g_1, \ldots, g_n\}$ of positive elements and every $g \in G$, there exist $\{n_k \}_{k=1}^{\infty}$ such that $g^{-n_k}g_ig^{n_k}$ is positive for $i =1, \ldots, n$.  For the proofs of this subsection, this will be the definition we use.  

\begin{proposition}
\label{no extensions}
The class of recurrent orderable groups is not closed under extensions.
\end{proposition}
\begin{proof}
Let $F$ be a finite index free subgroup of $\mathrm{SL}(2, \mathbb{Z})$ and $F \ltimes \mathbb{Z}^2$ the semidirect product arising from the natural action of $F$ on $\mathbb{Z}^2$.  From \cite[Example 4.6]{Morris06} this group admits no recurrent orderings, yet it fits into a short exact sequence
\[ 1 \rightarrow \mathbb{Z}^2 \rightarrow \mathbb{Z}^2  \ltimes F \rightarrow F \rightarrow 1
\]
where both ends of the sequence are recurrent orderable (in fact, bi-orderable).
\end{proof}

As mentioned, this means that the class of recurrent orderable groups does not satisfy Theorem \ref{classical BH}.  However if we strengthen the conditions on the terms in the short exact sequence, there is a version of this theorem that holds for recurrent orderable groups.

The proof below uses the notion of the \textit{positive cone} of an ordering.  That is, given an ordering $<$ of a group $G$, we can identify the given ordering with the set \[P = \{ g \in G \mid g>id \}.
\]
Conversely, any set $P \subset G$ satisfying $P \sqcup P^{-1} = G \setminus \{ id \}$ and $P \cdot P \subset P$ defines a left-ordering via the prescription $g<h$ if and only if $g^{-1}h \in P$.  The properties of being Conradian, recurrent, or bi-invariant can be translated into corresponding properties of positive cones; for instance the bi-orderings of a group $G$ correspond precisely to the positive cones $P$ satisfying $gPg^{-1} \subset P$ for all $g \in G$.

The set of all left-orderings of $G$ can therefore be identified with the corresponding set of positive cones in $G$, we denote this set by $\mathrm{LO}(G)$.  Similarly we define $\mathrm{BO}(G)$, the set of positive cones of bi-orderings of $G$.  Each of $\mathrm{LO}(G)$ and $\mathrm{BO}(G)$ is naturally a closed subset of the power set $\mathcal{P}(G)$ (for background, see \cite[Chapters 1 and 10]{CR16}), making each into a compact space.  The sets $V_g = \{ X \subset G \mid g \in X \}$ (where $g \in G$) form a subbasis for the topology on $\mathcal{P}(G)$, and thus a subbasis for the topology on $\mathrm{LO}(G)$ (resp. $\mathrm{BO}(G)$) is the collection of all sets $U_g = V_g \cap \mathrm{LO}(G)$ (resp. $U_g = V_g \cap \mathrm{BO}(G)$) where $g \in G$.

\begin{proposition}
Suppose that 
\[1 \rightarrow K \rightarrow G \stackrel{q}{\rightarrow} H \rightarrow 1
\]
is a short exact sequence of groups where $K$ is bi-orderable and $H$ is left-orderable, countable and amenable\footnote{Countable, amenable groups that are left-orderable admit recurrent orderings, by \cite{Morris06}.}.  Then $G$ admits a recurrent ordering.
\end{proposition}
\begin{proof}
Consider the action of $H$ on the space $\mathrm{BO}(K) \times \mathrm{LO}(H)$, where $\mathrm{BO}(K)$ and $\mathrm{LO}(H)$ are the spaces of bi- and left-orderings of $K$ and $H$ respectively.  Here, $H$ acts by outer automorphisms on $\mathrm{BO}(K)$ and by conjugation on $\mathrm{LO}(H)$. That is, if the action of $H$ on $K$ is given by $\psi: H \rightarrow \mathrm{Out}(K)$ where $\psi$ sends  each $h\in H$ to the outer automorphism $\phi_h$, then the action of $H$ on $\mathrm{BO}(K) \times \mathrm{LO}(H)$ is given by $h(P_K, P_H) = (\phi_h(P_K), h^{-1}P_Hh)$ for all $(P_K, P_H ) \in \mathrm{BO}(K) \times \mathrm{LO}(H)$.  Since every outer automorphism of $K$ acts on $\mathrm{BO}(K)$ as a homeomorphism and the conjugation action of $H$ on $\mathrm{LO}(H)$ is also an action by homeomorphisms, the action of $H$ on  $\mathrm{BO}(K) \times \mathrm{LO}(H)$ is by homeomorphisms.

Since $\mathrm{BO}(K) \times \mathrm{LO}(H)$ is a compact Hausdorff space and $H$ is amenable, there is a probability measure $\mu$ on $\mathrm{BO}(K) \times \mathrm{LO}(H)$ that is invariant under the $H$-action, meaning we can apply the Poincar\'{e} recurrence theorem to this action as in \cite{Morris06}.  As a result, there exists a point $(P_H, P_K) \in \mathrm{BO}(K) \times \mathrm{LO}(H)$ that satisfies:  For every $h \in H$ and every open set $U$ containing $(P_H, P_K)$ there exists an increasing sequence of positive integers $\{n_k\}_{k=1}^{\infty}$ such that $h^{n_k}(P_H, P_K) \in U$ for all $k$.  Set $P_G = P_K \cup q^{-1}(P_H)$, we next check that $P_G$ is the positive cone of a recurrent ordering.

Let $g_1, \ldots, g_n \in P_G$ be given.  Suppose that we have enumerated the $g_i$ so that $g_1, \ldots, g_r \in K$ and $q(g_i) \in P_H$ for $i \geq r+1$, where $1 \leq r \leq n$.  Now let $g \in G$ be given.  

If $g \in K$, then $g g_i g^{-1} \in P_K$ for $i =1, \ldots, r$ since $P_K$ is the positive cone of a bi-ordering. On the other hand for $j = r+1, \ldots, n$ we have $q(g g_j g^{-1}) = q(g_j) \in P_H$.  In either case, $g g_i g^{-1} \in P_G$ for all $i=1, \ldots, n$.

If $g \notin K$ then suppose $q(g) =h$ and consider the neighbourhood 
\[V = \bigcap_{i=1}^r U_{g_i} \cap \bigcap_{j=r+1}^n U_{g_j}
\]
of $(P_K, P_H)$ in $\mathrm{BO}(K) \times \mathrm{LO}(H)$. Choose $\{n_k\}_{k=1}^{\infty}$ such that 
\[h^{n_k}(P_K, P_H) = (\phi_h^{n_k}(P_K), g^{-n_k} P_H g^{n_k}) \in V
\] for all $k$.  Then $g_1, \ldots, g_r \in  \phi_h^{n_k}(P_K)$ and $g_{r+1}, \ldots, g_n \in g^{-n_k} P_H g^{n_k}$.  Noting that $\phi_h^{n_k}(P_K) = g^{-n_k} P_K g^{n_k}$ one finds that $g^{n_k} g_i g^{-n_k} \in P_G$ for all $i =1 , \ldots, n$, so the positive cone $P_G$ is recurrent.
\end{proof}

As a sample application of the previous proposition, we have the following.

\begin{corollary}
Suppose that $G$ and $H$ are countable amenable left-orderable groups.  Then the free product $G*H$ admits a recurrent ordering.
\end{corollary}
\begin{proof}
There is a well-known short exact sequence
\[ 1 \rightarrow F \rightarrow G*H \rightarrow G \times H \rightarrow 1
\]
where $F$ is a free group, hence bi-orderable.  Since $G \times H$ is left-orderable and amenable whenever $G$ and $H$ have these properties, the result follows.
\end{proof}

This prompts the following question.

\begin{question}
If $(G, <_G)$ and $(H, <_H)$ are groups equipped with recurrent orderings, does the free product $G*H$ admit a recurrent ordering that extends the orderings $<_G$ and $<_H$?
\end{question}

It also follows that many Conradian left-orderable $3$-manifold groups are in fact recurrent orderable.

\begin{corollary}
Suppose that $M$ is a $3$-manifold that fibres over the circle with fibre an orientable surface $\Sigma$.  Then $\pi_1(M)$ admits a recurrent ordering.
\end{corollary}
\begin{proof}
Since $\mathbb{Z}$ is left-orderable and amenable, and the fundamental group of every orientable surface is bi-orderable \cite{RW01}, the previous theorem applies to the homotopy exact sequence of the fibration 
\[1 \rightarrow \pi_1(\Sigma) \rightarrow \pi_1(M) \rightarrow \pi_1(S^1) \rightarrow 1.
\]
\end{proof}

Beyond these results, developing a Burns-Hale type theorem for the class of recurrent orderable groups seems particularly difficult.  The obstruction is twofold: First, it is unknown whether or not the property of admitting a recurrent ordering is a local property \cite[Question 3.42]{Navas10}, and second, the set of recurrent orderings is not a closed (hence not compact) subset of $\mathrm{LO}(G)$.

\begin{example}\cite[cf. Example 3.40]{Navas10}  Here is a simple example of a sequence of recurrent orderings converging to a non-recurrent ordering in the space of left orderings of a group $G$.  Consider the abelian group $\mathbb{Z}[t, t^{-1}]$ and let $\mathbb{Z} = \langle z \rangle$ act on $\mathbb{Z}[t, t^{-1}]$ by multiplication by $t$.  Form the semidirect product $G = \mathbb{Z}[t, t^{-1}] \ltimes \langle z \rangle$ and construct positive cones $Q_i \subset \mathbb{Z}[t, t^{-1}]$ as follows.

Given $\sum_{k=1}^r a_k t^{n_k}$, suppose that $n_1 < \ldots < n_r$ and that the $a_k$ are nonzero.  Suppose that $n_r = mi +j$ with $0 \leq j < i$ and $m$ an even integer.  Declare $\sum_{k=1}^r a_k t^{n_k} \in Q_i$ if $a_r >0$.  Otherwise, if $n_r = mi +j$ where $1 \leq j < i$ and $m$ is an odd integer, declare $\sum_{k=1}^r a_k t^{n_k} \in Q_i$ if $a_r <0$.  The positive cones $Q_i$ converge (as a sequence in the space $\mathrm{LO}(\mathbb{Z}[t, t^{-1}])$) to the positive cone 
\[ Q = \left\{ \sum_{k=1}^r a_k t^{n_k} \mid \mbox{$n_r \geq 0$ and $a_r > 0$ or $n_r<0$ and $a_r< 0$} \right\}.
\]

Now using the short exact sequence 
\[ 1 \rightarrow \mathbb{Z}[t, t^{-1}] \rightarrow G \rightarrow \langle z \rangle \rightarrow 1
\]
create a sequence of positive cones $P_i \subset G$ lexicographically, by setting $(\sum_{k=1}^r a_k t^{n_k}, z^m) \in P_i$ if $m>0$ or $m=0$ and $\sum_{k=1}^r a_k t^{n_k} \in P_i$.  Similarly create a positive cone $P$ using this lexicographic construction and the positive cone $Q$ on the subgroup $\mathbb{Z}[t, t^{-1}]$.

By construction each of the positive cones $P_i$ is recurrent (the orbit of each $P_i$ under the action of $G$ on $\mathrm{LO}(G)$ has $2i$ elements), yet their limit (as a sequence in $\mathrm{LO}(G)$) is the positive cone $P$, which is not recurrent.
\end{example}

It follows that the compactness arguments needed for a Burns-Hale type theorem (e.g. see Proposition \ref{pre order}) do not apply to the set of recurrent orderings, as they do in the cases of left, Conradian, circular and bi-orderings.  

\subsection{Circularly orderable groups.}
For a group $G$, if $t \in G^3$ and $y \in G$ we will use the notation $y \cdot t$ to indicate the component-wise multiplication of a group element on triples. 

Recall that a group $G$ is \textit{circularly orderable} if and only if there exists a function $c: G^3 \rightarrow \{ \pm1, 0\}$ satisfying:
\begin{enumerate}
\item $c(x_1, x_2, x_3) = 0$ if and only if $x_i = x_j$ for some $i \neq j$,
\item $c$ satisfies a cocyle condition:
\[c(x_1, x_2, x_3) - c(x_1, x_2, x_4) + c(x_1, x_3, x_4) - c(x_2, x_3, x_4) = 0
\]
for all $x_1, x_2, x_3, x_4 \in G$,
\item $c$ is left-invariant: For every triple $t \in G^3$ we have
\[ c(y \cdot t) = c(t)
\]
for all $y \in G$.
\end{enumerate}

Any function $c$ satisfying the conditions above is called a \textit{left circular ordering} of $G$ (hereafter shortened to ``circular ordering"), and we denote the set of all such functions by $\mathrm{CO}(G)$.   The set $\mathrm{CO}(G)$ is a subset of $\{ 0, \pm1\}^{G^3}$, and if we equip  $\{ 0, \pm1\}$ with the discrete topology and $\{ 0, \pm1\}^{G^3}$ with the product topology, the subspace topology inherited by $\mathrm{CO}(G)$ makes it into a compact space (for a thorough introduction to $\mathrm{CO}(G)$, see \cite{BS15}).  For each triple of group elements $t = (g_1, g_2, g_3) \in G^3$ and for each $i \in \{ \pm 1, 0\}$, set 
\[ U_t^i =  \{ \phi: G^3 \rightarrow \{ \pm1, 0\} \mid \phi(t) = i \}.
\]
The collection of all such sets form a subbasis for the topology on $\mathrm{CO}(G)$.

We use this topology to develop a condition on finite subsets of $G$ that will guarantee that $G$ is circularly orderable.  The techniques follow a similar development of ideas found in \cite{Conrad59, Navas10, CR16} in the cases of left, Conradian and bi-ordered groups.

\begin{definition}
Let $G$ be a group with generating set $S$ and $G_k$ the set of all words of length less than or equal to $k\in \mathbb{N}$.  A \textit{length $k$ circular pre-order} is a function $c:G_k^3 \rightarrow \{ 0, \pm 1 \}$ satisfying:
\begin{enumerate}
\item $c(x_1, x_2, x_3) = 0$ if and only if $x_i = x_j$ for some $i \neq j$,
\item $c$ does not violate the cocycle condition, 
\item if $g \in G_k$ and $t \in G_k^3$ then $c(g \cdot t) = c(t)$ whenever $g \cdot t \in G_k^3$.
\end{enumerate}
\end{definition}

Before the next proposition we introduce some notation that will be useful during the proof.  For a group set $X$, the \textit{big diagonal} of $X^3$ is the set $\Delta(X^3) = \{ (x_1, x_2, x_3) \mid x_i = x_j \mbox{ for some } i \neq j \}$.  

\begin{proposition}
\label{pre order}
A group $G$ is circularly orderable if and only if it admits a length $k$ circular pre-order for every $k \geq 1$.
\end{proposition}
\begin{proof}
If $G$ is circularly orderable it obviously admits length $k$ pre-orders for all $k$ by restricting any given circular ordering to $G_k^3$.

On the other hand, suppose that $G$ admits a length $k$ circular pre-order for every $k \geq 1$.  For each $k \geq 1$, set
 \[P_k = \{c : G^3 \rightarrow \{0, \pm1\} \mid c|_{G_k^3} \mbox{is a circular pre-order}. \}\]
 One can check that $P_k$ is closed in $\{0, \pm 1 \}^{G^3}$, for example let us consider condition (1) above.  Set $\Delta_k = \Delta( G_k^3)$.  A function $c$ violates condition (1) if and only if there exists a triple $t \in \Delta_k$ such that $c(t) =  \pm 1$ or a triple $t \in G_k^3 \setminus \Delta_k$ such that $c(t) = 0$.  So the set of functions $c :G^3 \rightarrow \{ 0, \pm1 \}$ that are \textit{not} length $k$ circular pre-orders is
\[ \bigcup_{t \in \Delta_k} U_t^1 \cup \bigcup_{t \in \Delta_k} U_t^{-1} \cup \bigcup_{t \in G_k^3 \setminus \Delta} U_t^0,
\]
an open set.  So condition (1) defines a closed subset of $\{0, \pm 1 \}^{G^3}$, as do conditions (2) and (3).

Now note that $P_{k+1} \subset P_k$ for all $k$, and thus $\bigcap_{k=1}^{\infty} P_k$ is an intersection of nested closed subsets of a compact space, hence nonempty. Any element in $\bigcap_{k=1}^{\infty} P_k$ is a circular ordering of $G$, and in fact one can check that $\bigcap_{k=1}^{\infty} P_k = \mathrm{CO}(G)$.
\end{proof}

The next result is well-known, we include a proof for the sake of completeness \cite[cf. Theorem 1.44]{CR16}.

\begin{lemma}
\label{fg circ}
A group $G$ is circularly orderable if and only if each of its finitely generated subgroups is circularly orderable.
\end{lemma}
\begin{proof}  Suppose every finitely generated subgroup admits a circular ordering.  For each finite subset $F \subset G$, set
\[ \mathcal{Q}(F) = \{ c: G^3 \rightarrow \{0, \pm 1\}  \mid  c|_{\langle F \rangle} \mbox{ is a circular ordering of $\langle F \rangle$} \}. 
\]
One checks that $\mathcal{Q}(F)$ is a closed nonempty subset $\{ 0, \pm 1 \}^{G^3}$.  The collection of all sets $\mathcal{Q}(F)$ has the finite intersection property, since for any collection $F_1, \ldots, F_n$ of finite subsets
\[ \mathcal{Q}(F_1 \cup \ldots \cup F_n) \subset \bigcap_{i=1}^n \mathcal{Q}(F_i).
\]
Thus $\bigcap_{F \subset G \tiny{\mbox{ finite}}} \mathcal{Q}(F)$ is nonempty by compactness, the elements of which are precisely circular orderings of $G$.
\end{proof}

For a collection of triples $T \subset G^3$, let $comp(T) = \bigcup_{i=1}^3 p_i(T)$, where $p_i : G^3 \rightarrow G$ is the $i$-th projection map (so it is the collection of all components of all triples in $T$).

\begin{proposition}
\label{circle exponents}
Let $G$ be a group, and $\Delta$ the big diagonal of $G^3$.  Then $G$ is circularly orderable if and only if for every finite set $T = \{t_1, \ldots, t_n \} \subset G^3 \setminus \Delta$ there exists $\epsilon_i = \pm1$ for $i =1,\ldots, n$ such that $c: T \rightarrow \{0, \pm1\}$ defined by $c(t_i) = \epsilon_i$ satisfies:
\begin{enumerate}
\item $c$ does not violate the cocycle condition, and
\item for all $y \in comp(T)$ and $t \in T$ if $y \cdot t \in T$ then $c(t) = c(y \cdot t)$.
\end{enumerate}
\end{proposition}
\begin{proof}
If $G$ admits a circular ordering $c$, then given any finite set $\{t_1, \ldots, t_n \} \subset  G^3 \setminus \Delta$ we define $\epsilon_i = c(t_i)$.  This choice clearly satisfies the required properties. 

For the other direction, we proceed as follows.  Suppose that for every finite subset of $G$ there exist $\epsilon_i$ as in the statement of the proposition.  By Lemma \ref{fg circ} we can assume that $G$ is finitely generated.   Choose a finite generating set of $G$ and consider $T = G_k^3 = \{ t_1, \ldots, t_n\}$.  Choose exponents $\epsilon_i$ satisfying the hypotheses of the proposition.  The resulting function $c(t_i) = \epsilon_i$ for all $t_i \in G_k^3$ is a length $k$ pre-order, and the conclusion follows from Proposition \ref{pre order}.
\end{proof}

Recall that a semigroup $S \subset G$ is called \textit{antisymmetric} if $S \cap S^{-1} = \emptyset$.

\begin{theorem}
\label{circ theorem}
A group $G$ is circularly orderable if and only if for every finite subset $X \subset G$ there exists a homomorphism $\phi : \langle X \rangle \rightarrow C$ onto a circularly ordered group and an antisymmetric semigroup $S \subset \ker(\phi)$ such that  $X^{-1}X \cap \ker(\phi) \subset (S \cup S^{-1} \cup \{ id \})$.
\end{theorem}
\begin{proof}
For the nontrivial direction of the proof, let $T \subset G^3 \setminus \Delta(G^3)$ be any finite subset, and set $X = comp(T)$.  Choose a surjective homomorphism $\phi: \langle X \rangle \rightarrow C$ and let $d$ be the circular ordering of $C$ and $S$ a subsemigroup satisfying the hypotheses of the theorem.  Let $(g_1, g_2, g_3)  \in T$ and define a function $c:T \rightarrow \{ 0, \pm 1\}$ as follows:  
\begin{enumerate}
\item If $\phi(g_1), \phi(g_2), \phi(g_3)$ are distinct then set $$c(g_1, g_2, g_3) = d(\phi(x_1), \phi(x_2), \phi(x_3)).$$
\item If two of $\phi(g_1), \phi(g_2), \phi(g_3)$ are equal, then we may assume that $\phi(g_1) = \phi(g_2)$.  Then $g_1^{-1} g_2  \in S \cup S^{-1}$, set $c(g_1, g_2, g_3) = 1$ if $g_1^{-1} g_2 \in S$ and $c(g_1, g_2, g_3) = -1$ if $g_1^{-1} g_2 \notin S$.
\item If $\phi(g_1)= \phi(g_2)= \phi(g_3)$ then $g_1^{-1}g_2, g_2^{-1}g_3, g_1^{-1}g_3 \in S \cup S^{-1}$.  Define $c(g_1, g_2, g_3) = 1$ if $\{g_1^{-1}g_2, g_2^{-1}g_3, g_1^{-1}g_3\} \cap S$ contains an odd number of elements, otherwise $c(g_1, g_2, g_3) = -1$.
\end{enumerate}
It is a straightforward case argument to verify that the function $c$ defined above does not violate the cocycle condition.  Similarly, if $g \in comp(T)$ and $t \in T$ satisfy $g \cdot t \in T$ then verifying that $c(t) = c(g \cdot t)$ is a matter of checking cases.  The result now follows from Proposition \ref{circle exponents}.
\end{proof}

\begin{corollary}
A group $G$ is circularly orderable if and only if for every finite subset $X \subset G$ there exists a homomorphism $\phi : \langle X \rangle \rightarrow C$ onto a circularly ordered group such that $\langle X^{-1}X \cap \ker(\phi) \rangle$ is left-orderable.
\end{corollary}
\begin{proof}
If $P$ is the positive cone of a left-ordering of $\langle X^{-1}X \cap \ker(\phi) \rangle$, take $S = P$ and apply Theorem \ref{circ theorem}.
\end{proof}

\begin{corollary}
A group $G$ is circularly orderable if and only if for every finite subset $X \subset G$ there exists a homomorphism $\phi : \langle X \rangle \rightarrow C$ onto a circularly ordered group such that $\phi|_X$ is injective.
\end{corollary}
\begin{proof}
This follows immediately from Theorem \ref{circ theorem}, by noting that $S = \emptyset$ will suffice whenever $\phi|_X$ is injective.
\end{proof}

\bibliographystyle{plain}

\bibliography{BH}

\end{document}